\documentclass[11pt]{amsart}

\usepackage{amssymb, amsmath, textcomp, amsthm, amsfonts, bm}
\usepackage{bbm,dsfont}
\usepackage{color}
\usepackage{cancel}
\usepackage{hyperref}
\usepackage{graphicx}

\usepackage{lipsum}
\usepackage{comment}

\makeatletter
\newcommand\@avprod[2]{%
  {\sbox0{$\m@th#1\prod$}%
   \vphantom{\usebox0}%
   \ooalign{%
     \hidewidth
     \smash{\vrule height\dimexpr\ht0+1pt\relax depth\dimexpr\dp0+1pt\relax}%
     \hidewidth\cr
     $\m@th#1\prod$\cr
   }%
  }%
}
\newcommand{\avprod}{\mathop{\mathpalette\@avprod\relax}\displaylimits}

\title[An improved bilinear restriction estimate for paraboloid in $\R^3$]{an improved bilinear restriction estimate for the paraboloid in $\R^3$}
\author{Changkeun Oh}
\date{}

\def\R{\mathbb{R}}

\def\nint{\mathop{\diagup\kern-13.0pt\int}}

\def\T{\mathbb{T}}

\def\beq{\begin{equation}}
\def\endeq{\end{equation}}
\def\bg{\begin{gathered}}
\def\eg{\end{gathered}}

\def\mc{\mathcal}
\def\lesim{\lesssim}

\numberwithin{equation}{section}
\theoremstyle{plain}
\newtheorem{thm}{Theorem}[section]
\newtheorem{prop}[thm]{Proposition}

\newtheorem{lem}[thm]{Lemma}

\newtheorem{defi}[thm]{Definition}

\newtheorem*{conj*}{Conjecture}

\newtheorem*{openproblem*}{Open Problem}

\begin{document}
\maketitle

\begin{abstract}
We obtain a sharp bilinear restriction estimate for the paraboloid in $\R^3$ for $q>13/4.$  
\end{abstract}

\section{Introduction}

Define an extension operator associated to the paraboloid in $\R^3$ by
\begin{equation}\label{extensionoperator}
Ef(x_1,x_2,x_3):=\int_{[-1,1]^2}f(\xi_1,\xi_2)
e\big(\xi_1 x_1+\xi_2 x_2+(\xi_1^2+\xi_2^2)x_3 \big)\,d\xi_1 d\xi_2
\end{equation}
for $f \in L^1([-1,1]^2)$.
Here, $e(a):=e^{2\pi i a}$ for $a \in \R$.
We say that two functions $f_1$ and $f_2$ are \textit{separated} provided that
\begin{equation}\label{separated}
    \mathrm{dist}(\mathrm{supp}(f_1),\mathrm{supp}(f_2) ) \gtrsim 1. 
\end{equation}
It is conjectured by Tao, Vargas and Vega in \cite{MR1625056} that the following bilinear restriction estimate
\begin{equation}\label{bilinear}
\||Ef_1Ef_2|^{1/2}\|_{L^q(\mathbb{R}^3)} \leq C_{p,q}\big(\|f_1\|_{L^p([-1,1]^2)}\|f_2\|_{L^p([-1,1]^2)}\big)^{\frac12}
\end{equation}
holds true for every pair of  separated $f_1$ and $f_2$ if and only if
\begin{equation}\label{bilinearregion}
    q \geq 3, \;\;\;\;\; \frac{5}{q}+\frac{3}{p} \leq 3, \;\;\; \text{and} \;\;\; \frac{5}{q}+\frac{1}{p} \leq 2.
\end{equation}

Our main theorem is as follows.

\begin{thm}\label{bilinearrestriction}
For every pair $p,q$ satisfying
\begin{equation}\label{goalpair}
    q>13/4, \;\;\;\;\; \frac{5}{q}+\frac{3}{p} < 3, \;\;\; \text{and} \;\;\; \frac{5}{q}+\frac{1}{p} < 2,
\end{equation}
it holds that
\begin{equation}
    \||Ef_1Ef_2|^{1/2}\|_{L^q(\mathbb{R}^3)} \leq C_{p,q}\big(\|f_1\|_{L^p([-1,1]^2)}\|f_2\|_{L^p([-1,1]^2)}\big)^{\frac12}
\end{equation}
for every pair of separated functions $f_1$ and $f_2$. Here the constant $C_{p,q}$ depends on $p,q$ and the implied constant in \eqref{separated}.
\end{thm}

The bilinear restriction problem is strongly tied to the restriction problem. The restriction conjecture states that the estimate
\begin{equation}\label{linear}
    \|Ef\|_{L^q(\R^3)} \leq C_{p,q} \|f\|_{L^{p}([-1,1]^2)}
\end{equation}
holds true for every function $f$, if and only if 
\begin{equation}\label{linearregion}
    q>3, \;\;\; \text{and} \;\;\; \frac{2}{q}+\frac{1}{p}\leq 1.
\end{equation}

\begin{figure}[ht]
\centering
\begin{minipage}[b]{0.48\linewidth}
\includegraphics[width=6.5cm]{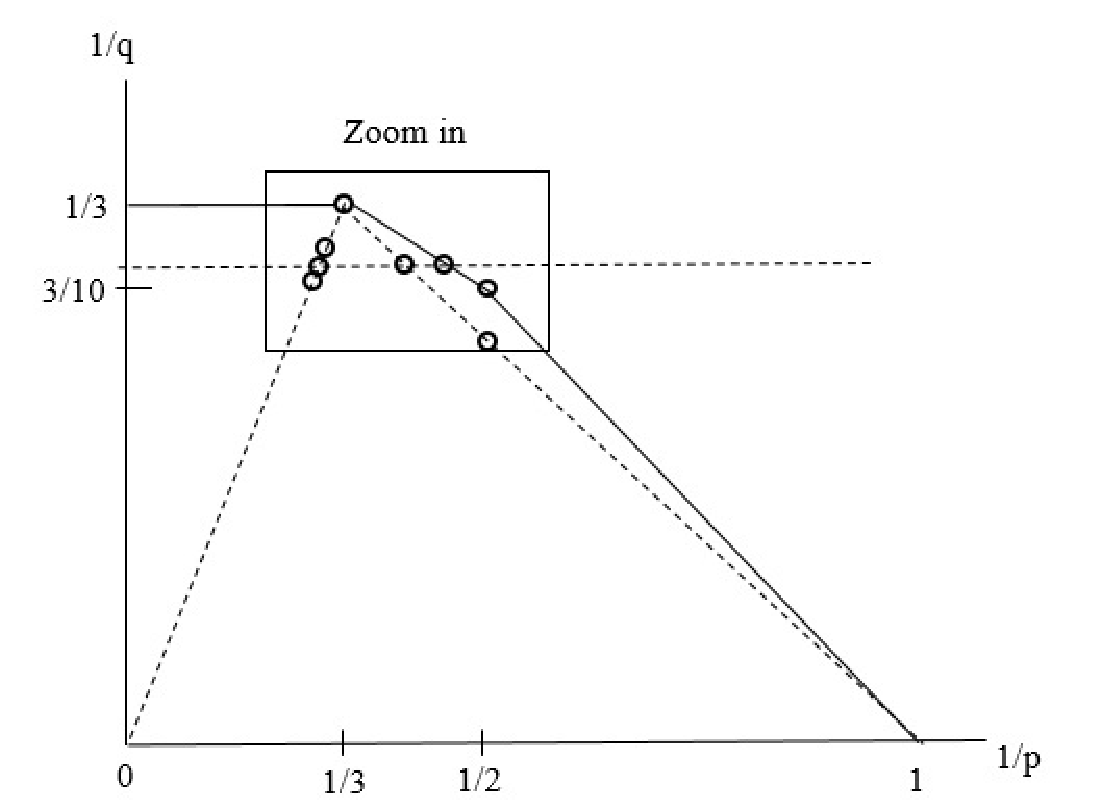}
\caption{}
\label{fig:minipage1}
\end{minipage}
\quad
\begin{minipage}[b]{0.48\linewidth}
\includegraphics[width=8cm]{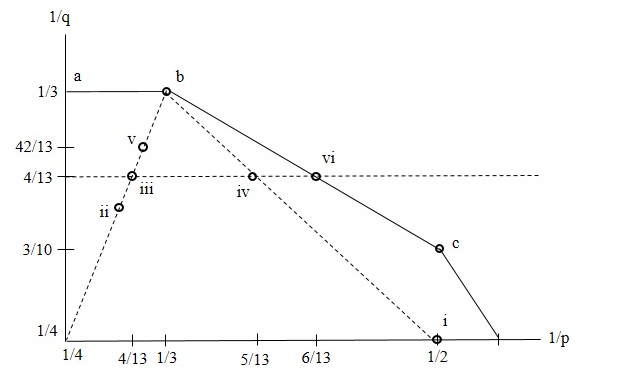}
\caption{Zoomed in}
\label{logo}
\end{minipage}
\end{figure}

The region \eqref{linearregion} is the trapezoidal region bounded by the points $a,b$ in Figure 2, $(0,0)$, and $(1,0)$\footnote{The points $(0,0)$ and $(1,0)$ are not represented in Figure 2.}, except for the upper line between $a$ and $b$ inclusive. The region \eqref{bilinearregion} is the pentagonal region bounded by the points $a,b,c$, $(0,0)$, and $(1,0)$, including the upper line mentioned previously. Note that the region \eqref{bilinearregion} is wider than the region \eqref{linearregion}.

The connection between the bilinear restriction estimate and the restriction estimate was discovered by Tao, Vargas, and Vega in \cite{MR1625056}, where they proved that the bilinear restriction estimate for a pair $(p,q)$ on the region \eqref{linearregion} implies the restriction estimate for the same pair $(p,q)$. The converse is also true by a simple application of H\"{o}lder's inequality. 

Let us briefly mention the recent progress on the restriction and bilinear restriction problems for paraboloid in $\R^3$.
In 2003, Tao \cite{MR2033842} proved the sharp bilinear restriction estimate for the paraboloid in $\R^3$ for the pair $(p,q)=(2,10/3+\epsilon)$ for an arbitrary number $\epsilon>0$ (the point $c$)\footnote{In the paper, it is proved the sharp $L^2$ bilinear restriction estimate for the paraboloid in $\R^n$ up to the endpoint  for all the dimensions $n$.}, which implies the restriction estimate for the paraboloid for $q>10/3$ by \cite{MR1625056}. His proof is based on Wolff's two ends argument in \cite{MR1836285}, where Wolff proved the sharp bilinear restriction estimate for the cone. The (linear) restriction estimate of Tao is improved by Bourgain and Guth \cite{MR2860188} to the range $q>3.3$ (the point $ii$), where they introduced the multilinear technique and combined it with some Kakeya estimate to get a better restriction estimate. Recently, Guth \cite{MR3454378} improved the restriction estimate  to the range $q>13/4$ (the point \textit{iii}). More precisely, he introduced the notion of a broad function and proved a broad function estimate for $q>13/4$ by using polynomial partitioning. This broad function estimate is slightly weaker than the bilinear restriction estimate, but the argument of \cite{MR1625056} still works equally well with the broad function estimate, so he was able to prove the restriction estimate for the same range of $q$. The restriction estimate of Guth is extended to the point $iv$ by Shayya \cite{MR3694293} (see also \cite{Kim2017SomeRO}).
The most recent result  is due to Wang \cite{Wang2018ARE} (the point \textit{v}), where she proved the restriction estimate for $q>3+3/13$ by proving the broad function estimate for the same range of $q$. Her proof of the broad function estimate combines Wolff's two ends argument with polynomial partitioning. For earlier results, we refer to Tao, Vargas and Vega \cite{MR1625056}, in particular, Table 1 on page 969 of their paper. 

Our bilinear restriction theorem improves Tao's sharp bilinear restriction estimates (the point $c$) to the range $q>13/4$ (the point $vi$). Also, our theorem recovers the broad function estimate of Guth and the restriction estimate of Shayya by the arguments of \cite{MR1625056}.
It looks plausible to generalize our result to improved $(n-1)$-linear restriction estimate for paraboloid in $\R^n$ as all the tools used in this paper are still available in high dimensions. However, for the sake of readability, we focus only on the three dimension. 

One natural question is whether one can generalize this result to more general surfaces under certain conditions. Bejenaru introduced interesting curvature conditions in \cite{bejenaru2020optimal}, and he proved  sharp $L^2 \rightarrow L^{10/3}$ bilinear  restriction estimates for general surfaces in $\R^3$ satisfying the conditions. His conditions are so general that even a surface with a vanishing principal curvature (for example, a cone) satisfies the conditions. Interestingly, our theorem is not always true for surfaces satisfying his conditions. For example, \cite{MR1748920} proved that \eqref{bilinear} is not true for a cone for any pair $(p,q)$ with $q<10/3$ and $p=\infty$. 
\\

The proof of Theorem \ref{bilinearrestriction} is built on the arguments of \cite{MR3454378} and \cite{bejenaru2020optimal}. Let us compare our proof with Guth's proof and explain the obstacle of our problem. By the wave packet decomposition (Lemma \ref{wpd}), we can decompose the functions $Ef$ into wave packets $Ef_{T}$ so that each wave packet $Ef_{T}$ is ``essentially supported'' on the tube $T$. In the study of the restriction problem, Guth applied polynomial partitioning and reduced the problem to some lower dimensional restriction problem in the sense that all the ``significant'' wave packets $Ef_T$ are contained in a thin neighborhood of a variety. Then he proved the lower dimensional restriction estimate. However, in our bilinear restriction problem, since the bilinear operator involves two functions, even if we can still apply polynomial partitioning to the bilinear operator, it is difficult to make all of the significant wave packets of $Ef_1$ and $Ef_2$ be contained in a thin neighborhood of a variety. This is the main obstacle of our problem.

Here are our ideas. By following the arguments of Guth, we reduce to the situation where all the significant wave packets of $Ef_{1}$ are contained in a thin neighborhood of a variety. Let us pretend that the variety is a two-dimensional plane in this paragraph. Then we apply some pigeonholing argument to the wave packets of $Ef_2$ so that all the significant wave packets form some fixed angle between the tube and the variety. If the angle gets smaller, then our wave packets of $Ef_{2}$ get closer to the thin neighborhood of the variety, and it gets closer to the lower dimensional problem, which can be dealt with by following the argument of Guth. On the hand hand, if the angle gets larger, then we observe that an intersection of the tube and the thin neighborhood of the variety gets smaller. This geometric observation gives a better $L^2$-estimate than usual (see Lemma \ref{L2estimate}). We quantify these two observation and combine them so that no matter what the angle is it gives the desired estimate. However, there is one additional issue: If the angle between a tube and a variety is ``almost'' perpendicular, then it is too far from a lower dimension situation and we cannot imitate the argument of Guth. In this case, we apply polynomial partitioning one more time as in \cite{bejenaru2020optimal}, and this takes care of the case (see a high angle dominant case (Subsection \ref{highangle})).

\subsection{Notation}
For each ball $B_R$ of radius $R$ with the center $c(B_R)$,
define the weight function 
\begin{equation}
    w_{B_R}(x):=\Big(1+\Big|\frac{x-c(B_R)}{R}\Big|\Big)^{-100},
\end{equation}
and the weighted integral
\begin{equation}
    \|F\|_{L^p(w_{B_R})}:=\Big( \int |F(x)|^pw_{B_R}(x)\,dx \Big)^{1/p}
\end{equation}
for every function $F \in L^{\infty}(\R^3)$.
For every measurable set $A$ with positive Lebesgue measure, we define the averaged $L^2$ integral by
\begin{equation}
    \|f\|_{L^2_{\mathrm{avg}}(A)}:= \Big( \frac{1}{|A|} \int_{A}|f|^2 \Big)^{1/2}.
\end{equation}
Note that 
\begin{equation}\label{trivialbound}
\|f\|_{L^2_{\mathrm{avg}}(A)} \leq \|f\|_{L^{\infty}(A)}.
\end{equation}

We write $A(R) \leq \mathrm{RapDec}(R)B$ to mean that for any power $\beta$, there is a constant $C_{\beta}$ such that
\begin{equation}
    A(R) \leq C_{\beta}R^{-\beta}B \;\; \text{for all $R \geq 1$}.
\end{equation}
Let us introduce the notation $
\avprod_{i=1}^{2}{a_i}:=|a_1a_2|^{1/2}$.
For $f_1,f_2$ and some quantities $A(f_1,f_2)$ and $B(f_1,f_2)$, we write $A(f_1,f_2) \lessapprox B(f_1,f_2)$ to mean that
\begin{equation}\label{lessapprox}
    A(f_1,f_2) \leq CB(f_1,f_2) + \mathrm{RapDec}(R)\avprod_{i=1}^{2}{\|f_i\|_2}.
\end{equation} Here, the constant $C$ is independent of $f_1, f_2$ and $R$. Note that this notation is not conventional.

For two non-negative numbers $A_1$ and $A_2$, we write $A_1 \lesim A_2$ to mean that there exists a constant $C$ such that $A_1 \le C A_2$. Similarly, we use $O(A_1)$ to denote a number whose absolute value is smaller than $CA_1$ for some constant $C$. We write $A_1 \simeq A_2$ if $A_1 \lesssim A_2$ and $A_2 \lesssim A_1$.
We also write $A_1 \ll A_2$ if $CA_1 \leq A_2$ for some sufficiently large number $C$.

For every set $S \subset \R^3$ and number $\rho>0$, we denote by $N_{\rho}(S)$ the $\rho$-neighborhood of $S$.
For every polynomial $P$, we denote by $\mc{Z}(P)$ the zero set of the polynomial $P$.
We introduce two parameters $\epsilon>0$ and $\delta>0$. The parameter $\delta>0$ will be much smaller than $\epsilon$. 

\subsection*{Acknowledgements}

Part of this work was done under the support of the NSF grant DMS-1800274. The author would like to thank his advisor Shaoming Guo for valuable comments. The author also would like to thank Youngwoo Koh for introducing the paper \cite{MR4205111} to him. The author also would like to thank Shengwen Gan for introducing the paper \cite{MR1748920}, where a counterexample for a bilinear restriction estimate for cone in $\mathbb{R}^3$ is constructed. The author also would like to thank the referees for carefully reading the manuscript and giving valuable suggestions.

\section{Preliminaries}\label{prelim}

We review a wave packet decomposition. Let us define some notation first. We decompose the square $[-1,1]^2$ into the dyadic squares $\theta$ of side length $R^{-1/2}$. Let $w_{\theta}$ denote the bottom left corner of $\theta$. Let $v_{\theta}$ denote the normal vector to the paraboloid at the point $(w_{\theta},|w_{\theta}|^2)$. We denote by $\mathcal{P}(R^{-1/2})$ the collection of the squares. Let $\T(\theta)$ denote a set of tubes covering $B_R \subset \R^3$, that are parallel to $v_{\theta}$ with radius $R^{1/2+\delta}$ and length $CR$. Denote $\T:=\cup_{\theta \in \mathcal{P}(R^{-1/2}) }\T(\theta)$. For each $T \in \T(\theta)$, let $v(T)$ denote the direction $v_{\theta}$ of the tube.

\begin{lem}[Wave packet decomposition]\label{wpd}
If $f \in L^2([-1,1]^2)$ then for each $T \in \mathbb{T}$ we can choose a function $f_T$ so that the following holds true:
\begin{enumerate}
    \item If $T \in \mathbb{T}(\theta)$ then $\mathrm{supp}(f_{T}) \subset 3\theta$;
    \item  If $x \in B_R \setminus T$, then $|Ef_T(x)| =\mathrm{RapDec}(R)\|f\|_2$;
    \item  For any $x \in B_R$, $|Ef(x)-\sum_{T \in \mathbb{T}}Ef_T(x)|= \mathrm{RapDec}(R)\|f\|_{L^2}$;
    \item If $T_1,T_2 \in \mathbb{T}(\theta)$ and $T_1,T_2$ are disjoint, then $|\int f_{T_1}\bar{f}_{T_2}| =\mathrm{RapDec}(R) \int_{\theta}|f|^2$;
    \item $\sum_{T \in \mathbb{T}(\theta)} \int_{[-1,1]^2}|f_T|^2 \lesssim \int_{\theta}|f|^2$.
\end{enumerate}

\end{lem}

This is the formulation of the wave packet decomposition in {\cite{MR3454378}}. We refer to Proposition 2.6 of \cite{MR3454378} for the proof; see Lemma 4.1 of \cite{MR2033842} and Lemma 2.2 of \cite{MR2218987} for another formulation. The functions $Ef_T$ are called \textit{wave packets}.
We need some $L^2$-orthogonality of wave packets. Here is one version of it. We refer to Lemma 2.7 and 2.8 of \cite{MR3454378} for the proof.

\begin{lem}[$L^2$-orthogonality]\label{l2orthogonality}
For any subset $\T_i \subset \T$, square $\theta \in \mathcal{P}(R^{-1/2})$ and function $f$, it holds that
\begin{equation}
\begin{split}
    \int_{3\theta}\Big|\sum_{T \in \T_i} f_T\Big|^2 &\lesssim \sum_{T \in \T_i}\int_{10\theta} |f_T|^2+\mathrm{RapDec}(R)\|f\|_{L^2(10\theta)}  
    \\&
    \lesssim \int_{20\theta}|f|^2
\end{split}
\end{equation}
and
\begin{equation}
    \int_{[-1,1]^2}\Big|\sum_{T \in \T_i} f_T\Big|^2 \lesssim \int_{[-1,1]^2}|f|^2.
\end{equation}
\end{lem}

Our proof of Theorem \ref{bilinearrestriction} relies on  polynomial partitioning. The interested readers should consult the introduction of \cite{MR3454378} for the historical backgrounds on  polynomial partitioning. 

We first introduce some terminology. For every polynomial $P:\R^n \rightarrow \R$, we denote by $\mc{Z}(P)$ the zero set of the polynomial $P$, and by $\mathrm{cell}(P)$ the collection of the connected components of $\R^n \setminus \mc{Z}(P)$. A set $\mc{Z}(P_1,\ldots,P_{n-m}):=\bigcap_{i=1}^{n-m} \mc{Z}(P_i)$ is called an $m$-dimensional transverse complete intersection if it satisfies
\begin{equation}
\bigwedge_{j=1}^{n-m}
    \nabla P_j(z) \neq 0
\end{equation}
for all $z \in \mc{Z}(P_1,\ldots,P_{n-m})$. A degree of the transverse complete intersection $\mc{Z}(P_1,\ldots,P_{n-m})$ is defined as the maximum of the degrees of $P_i$. This definition of the degree is non-standard in the sense that the set depends on the choice of polynomials defining the variety. It might be possible to define the degree in a more natural way, but this definition does not make a trouble in our application, so we use this definition.

The following is the polynomial partitioning lemma used in \cite{guth2018} (see also Section 6 of \cite{HR2019}).

\begin{lem}[Polynomial partitioning lemma]\label{polypartitioning} Let $1 \leq m \leq n$ and $d \geq 0$.
Let $F$ be a non-negative $L^1$ function on $\R^n$  supported on $B_{R} \cap N_{R^{1/2+\delta}}(\mc{Z})$ where $\mc{Z}:=\mc{Z}(P_1,\ldots,P_{n-m})$ is an $m$-dimensional transverse complete intersection of degree at most $d$. Then at least one of the following holds:  
\begin{enumerate}
    \item 
    There exists a polynomial $P:\R^n \rightarrow \R$ of degree at most $O(d)$ with the following properties:
    \begin{itemize}
        \item $\# \mathrm{cell}(P) \simeq d^m$.
        \item For every $O' \in \mathrm{cell}(P)$, we define the cells $O:=O' \setminus N_{R^{1/2+\delta}}(\mc{Z}(P))$. Then there exists a subcollection $\mathrm{cell}_{\circ}(P)$ of $\mathrm{cell}(P)$ such that for every $O$ generated by $O' \in \mathrm{cell}_{\circ}(P)$
        \begin{equation}
           \int_{\R^n}F \simeq d^m \int_{O} F.
        \end{equation}
        Moreover, the number of the cells $O$ generated by $\mathrm{cell}_{\circ}(P)$ is comparable to $d^m$. 
    \end{itemize}
        \item There exists an $(m-1)$-dimensional transverse complete intersection $\mc{Z}_1$ of degree at most $d$ such that
        \begin{equation}
            \int_{\R^n }F
            \lesssim
            \int_{B_R \cap N_{R^{1/2+\delta}}(\mc{Z}_1)}F.
        \end{equation}
\end{enumerate}
\end{lem}

Since we use the polynomial method and the wave packet decomposition, it is necessary to understand the interplay between a variety and tubes. In particular, we need to answer two questions:
\begin{itemize}
    \item Describe the intersection of a tube and a thin neighborhood of a variety.
    \item How many tubes pointing in  ``separated'' directions can be contained in a thin neighborhood of a variety?
\end{itemize}
The author learned the answer of the first question in \cite{MR4205111}, where Zahl uses a result of Basu, Pollack, and Roy in \cite{MR1401711} and obtains a satisfactory answer (see \eqref{zahlobservation}). The second question is answered by \cite{MR3454378} in three dimensions, by \cite{MR3820441} in four dimensions, and by \cite{MR3881832} in all the dimensions, whose results are called the polynomial Wolff axioms. We will use Lemma \ref{polylemma}, which is a slightly more general version of them.

\begin{defi}[Semi-algebraic set]
A set $S \subset \R^n$ is called semi-algebraic if it can be written as a finite union of sets of the form
\begin{equation}
    \{x \in \R^n: P_1(x)>0, \ldots, P_k(x)>0,P_{k+1}(x)=\ldots=P_{k+l}(x)=0   \},
\end{equation}
where $P_1,\ldots,P_{k+l}$ are polynomials. A union of such sets is called a presentation of $S$. The complexity of a presentation is the sum of the degrees of the polynomials. The complexity of a semi-algebraic set $S$ is the minimum complexity of its presentation.
\end{defi}

\begin{lem}[\cite{MR1401711}, cf. Theorem 2.3 of \cite{MR4205111}]\label{bpr96}
Let $S \subset \R^n$ be a semi-algebraic set of complexity $D$. Then there exists a constant $C(n,D)$ so that $S$ has at most $C(n,D)$ connected components.
\end{lem}

\begin{lem}[Lemma 2.11 of \cite{MR4205111}, cf. Theorem 1.4 of \cite{hickman2019improved}]\label{polylemma}
Let $n,E$ and $K$ be integers with $n \geq 2$, and let $\epsilon>0$. Then there is a constant $C(n,E,K,\epsilon)>0$ so that the following holds. Let $Z$ be a semi-algebraic set of complexity at most $E$. Let $r>0$. Suppose that $Z \subset \R^n $ has diameter $r$ and obeys 
\begin{equation}\label{wongkew}
    |N_{\rho}(Z)\cap B(x,r)| \leq E\rho r^{n-1} \text{ for all balls } B(x,r).
\end{equation}
Let $0<\delta< \rho/r$, and let $\mathcal{L}$ be a set of lines in $\R^n$ pointing in $\delta$-separated directions with the property that for each $L \in \mathcal{L}$ 
\begin{center}
    $L \cap N_{\rho}Z$ contains a line segment of length $r/K$.
\end{center}
Then 
\begin{equation}
    \# \mathcal{L} \leq C(n,E,K,\epsilon)
    \big(\frac{r}{\rho}\big)^{-1+\epsilon}\delta^{1-n-\epsilon}.
\end{equation}
\end{lem}

\section{A proof of  theorem \ref{bilinearrestriction}: polynomial partitioning}

Theorem \ref{bilinearrestriction} can be deduced from the following:
\begin{prop}\label{reducedmainestimate}
For every $\epsilon>0$, it holds that
\begin{equation}\label{prop31}
\begin{split}
&\Big\|\avprod_{i=1}^{2}|Ef_i|\Big\|_{L^{13/4}(B_R)} 
\\&
\leq C_{\epsilon}R^{10\epsilon} \big(\avprod_{i=1}^{2}\|f_i\|_{L^2}\big)^{\frac{12}{13}+\epsilon}
\big(\avprod_{i=1}^{2}\max_{\theta \in \mathcal{P}(R^{-1/2}) }\|f_i\|_{L^{2}_{\mathrm{avg}}(\theta)}\big)^{\frac{1}{13}-\epsilon}
\end{split}
\end{equation}
for every $R \geq 1$ and separated functions $f_1$ and $f_2$.
\end{prop}

Let us assume the above proposition and prove Theorem \ref{bilinearrestriction}.
Recall that Tao \cite{MR2033842} proved a sharp bilinear restriction estimate for $(p,q)=(2,10/3+\epsilon)$ for arbitrary $\epsilon>0$. By the trivial estimate \eqref{trivialbound}, the average norm in \eqref{prop31} can be bounded by $L^{\infty}$-norm. Applying the resulting inequality with $f_i=\chi_{F_i}$ for separated measurable sets $F_i \subset [-1,1]^2$ we obtain 
\begin{equation}
    \||E\chi_{F_1}E\chi_{F_2}|^{1/2}\|_{L^{13/4}(B_R)} \leq C_{\epsilon} R^{10\epsilon} \avprod_{i=1}^2|F_i|^{6/13+\epsilon/2}.
\end{equation}
By applying H\"{o}lder's inequality and using Tao's result, we obtain
\begin{equation}\label{bilinearinterpolation}
    \||Ef_1Ef_2|^{1/2}\|_{L^q(B_R)} \leq C_{p,q,\epsilon} R^{10\epsilon}\big(\|f_1\|_{L^p}\|f_2\|_{L^p}\big)^{1/2}
\end{equation}
for every pair $(p,q)$ satisfying \eqref{goalpair} and every $f_i=\chi_{F_i}$. We apply Lemma 1.4.20 of Grafakos's book \cite{MR3243734} componentwise to the bilinear operator, and obtain \eqref{bilinearinterpolation} for every pair $(p,q)$ satisfying \eqref{goalpair} and every $f_i \in L^p$. Applying the epsilon-removal lemma (Lemma 2.4 in \cite{MR1748920}) completes the proof of Theorem \ref{bilinearrestriction}.
\\

In the rest of the paper, we focus on the proof of Proposition \ref{reducedmainestimate}.
Our proof relies on the induction on scales. 
Let us fix $\epsilon>0$. We may assume that $\epsilon$ is sufficiently small. We take $C_{\epsilon}$ large enough so that \eqref{prop31} trivially holds true for small $R$. Hence, it suffices to consider a large $R$ and prove \eqref{prop31} under the assumption that it holds true for $<R/2$. We record our induction hypothesis.

\subsection*{Induction hypothesis:}
\begin{equation}\label{inductionhypothesis}
    \eqref{prop31} \text{ holds true all the radii smaller than}  \leq R/2.
\end{equation}

In order to close the induction, it is important to keep in mind that we need to prove \eqref{prop31} with the same constant $C_{\epsilon}$. The constant $C_{\epsilon}$ will not vary from line-to-line.

\subsection{Polynomial partitioning} 
Let $\delta>0$ be some number much smaller than $\epsilon$.
Let $D$ be a sufficiently large number independent of $R$, which will be determined later. 
We apply the polynomial partitioning lemma (Lemma \ref{polypartitioning}) to the function $|Ef_1Ef_2|^{1/2}$ and $\mc{Z}=\R^3$. Then there are two possibilities:

\subsection*{The cellular case}
There exists a polynomial $P$ of degree at most $D$ such that
\begin{equation}
\R^3 \setminus \mc{Z}(P) = \bigsqcup_{k=1}^{M} O_k',
\end{equation}
where $M \simeq D^3$, and $O_k'$ is a connected component of $\R^3 \setminus \mc{Z}(P)$, 
and if we define the cells $O_k:=B_R \cap \big(O_k' \setminus N_{R^{1/2+\delta}}(\mc{Z}(P))\big)$, then
\begin{equation}\label{equalcontribution}
\||Ef_1Ef_2|^{1/2}\|_{L^{13/4}( B_R)}^{13/4} \simeq D^3\||Ef_1Ef_2|^{1/2}\|_{L^{13/4}(O_{k} )}^{13/4}
\end{equation}
for $\simeq D^3$ many cells $O_k$. 

\subsection*{The wall case}
There exists a two-dimensional transverse complete intersection $\mc{Z}(P_1)$ of degree at most $D$ such that
\begin{equation}
    \||Ef_1Ef_2|^{1/2}\|_{L^{13/4}( B_R)} \lesssim \||Ef_1Ef_2|^{1/2}\|_{L^{13/4}(B_R \cap N_{R^{1/2+\delta}}(\mc{Z}(P_1) )) }.
\end{equation}

It is well-known that a transverse complete intersection can be thought of as a smooth manifold locally. This will enable us to define a tangent plane on a point of the transverse complete intersection.

\subsection{The cellular case}\label{cellcase}
In this subsection, we will consider the cellular case and prove \eqref{prop31}.
This case can be dealt with by following the arguments of \cite{MR3454378} line by line. We include the details for the completeness of the paper.
 Recall that $\T$ is a collection of the tubes defined at the beginning Section \ref{prelim}. By abuse the notation, we pretend that cells always indicate the cells $O_k$ satisfying \eqref{equalcontribution}. In this subsection, the constant $C$ may vary from line-to-line. This constant $C$ is independent of the parameters $\epsilon$, $D$ and $R$.

We first note that,
by property (2) and (3) of Lemma \ref{wpd},
on each cell $O_k$
\begin{equation}
    Ef_i=\sum_{T \in \T: T \cap O_k \neq \emptyset }Ef_{i,T}+\mathrm{RapDec}(R)\|f_i\|_2.
\end{equation}
For simplicity, we introduce the notation $f_{i,O_k}:= \sum_{T \in \T: T \cap O_k \neq \emptyset } f_{i,T}$. By the equality above, for every $k$, it holds that
\begin{equation}\label{cell:step2}
\begin{split}
    \||Ef_1Ef_2|^{1/2}\|_{L^{13/4}(B_R)} \lessapprox 
    D^\frac{12}{13}\| |Ef_{1,O_k}Ef_{2,O_k}|^{1/2}\|_{L^{13/4}(O_{k})},
\end{split}
\end{equation}
where the notation $\lessapprox$ is introduced in \eqref{lessapprox}.
Notice that, by the fundamental theorem of algebra, each tube $T \in \T$ passes through at most $D+1$ cells $O_k$  (see Lemma 3.2 of \cite{MR3454378}). Hence, by the orthogonality of wave packets (Lemma \ref{l2orthogonality}), it holds that
\begin{equation}\label{38}
\begin{split}
    \sum_{O_k}\|f_{i,O_k}\|_2^2
    &\leq C\sum_{O_k}\sum_{T_i \in \T: T_i \cap O_k \neq \emptyset }\|f_{i,T_i}\|_2^2
    \\&
    \leq 
    C\sum_{T_i \in \T}\sum_{O_k: O_k \cap T_i \neq \emptyset }\|f_{i,T_i}\|_2^2
    \leq 2CD \|f_i\|_2^2
\end{split}
\end{equation}
for some constant $C$. It is straightforward to see that $\frac{9}{10}\#O_k$ many cells $O_k$ satisfy the following:
\begin{equation}
    \|f_{i,O_k}\|_2^2 \leq  \big(\frac{100}{\# O_k}\big) C D \|f_i\|_2^2.
\end{equation}
Thus, by recalling that $\#O_k \simeq D^3$ and by pigeonhling, we can choose a cell $O_{k_0}$ such that
\begin{equation}\label{goodchoice}
    \|f_{i,O_{k_0}}\|_2^2 \lesssim D^{-2} \|f_i\|_2^2.
\end{equation}
for both $i=1,2$.
Let us fix such $O_{k_0}$ and decompose $O_{k_0}$ into smaller balls of radius at most $R/2$ and apply the induction hypothesis \eqref{inductionhypothesis} to the right hand side of \eqref{cell:step2} on each smaller ball $B_{R/2}$. Then the left hand side of \eqref{cell:step2} is bounded by
\begin{equation}
\begin{split}
CC_{\epsilon}
    D^{\frac{12}{13}} R^{ 10\epsilon } \big(\avprod_{i=1}^{2}\|f_{i,O_{k_0}}\|_{L^2}\big)^{\frac{12}{13}+\epsilon }
\big(\avprod_{i=1}^{2}\max_{\theta}\|f_{i,O_{k_0}}\|_{L^{2}_{\mathrm{avg}}(\theta)}\big)^{\frac{1}{13}-\epsilon }.
\end{split}
\end{equation}
We now apply \eqref{goodchoice} to the $L^2$-norm and the $L^2$-orthogonality (Lemma \ref{l2orthogonality}) to the $L^2_{\mathrm{avg}}$-norm. Then the above term is further bounded by 
\begin{equation}
CC_{\epsilon}
    D^{-\epsilon}R^{ 10\epsilon} \big(\avprod_{i=1}^{2}\|f_i\|_{L^2}\big)^{\frac{12}{13}+\epsilon }
\big(\avprod_{i=1}^{2}\max_{\theta}\|f_i\|_{L^{2}_{\mathrm{avg}}(\theta)}\big)^{\frac{1}{13}-\epsilon}.
\end{equation}
It suffices to take $D$ sufficiently large so that $CD^{-\epsilon} \leq 1$, and thus, we can close the induction.
\\

To summarize, we have proved \eqref{prop31} in the cellular case. It remains to prove \eqref{prop31} in the wall case, which will be done in the following sections.

\section{A proof of  theorem \ref{bilinearrestriction}: The wall case}
In this section, we consider the wall case. We need to prove that under the induction hypothesis \eqref{inductionhypothesis}
\begin{equation}\label{wallcase}
\begin{split}
&\Big\|\avprod_{i=1}^{2}|Ef_i|\Big\|_{L^{13/4}(B_R \cap N_{R^{1/2+\delta}}(\mc{Z}(P_1) ))} 
\\&
\leq CC_{\epsilon}R^{10\epsilon} \big(\avprod_{i=1}^{2}\|f_i\|_{L^2}\big)^{\frac{12}{13}+\epsilon}
\big(\avprod_{i=1}^{2}\max_{\theta \in \mathcal{P}(R^{-1/2}) }\|f_i\|_{L^{2}_{\mathrm{avg}}(\theta)}\big)^{\frac{1}{13}-\epsilon}
\end{split}
\end{equation}
for some small constant $C$. Here, $\mc{Z}(P_1)$ is a two-dimensional transverse complete intersection of degree at most $D$. Recall that the constant $D$ is independent of the parameter $R$, and this fact will play a role in the proof.
\medskip

Let $D_1$ be a large constant compared to $D$ and be independent of the parameter $R$.
We consider two subcases according to whether there exists an one-dimensional transverse complete intersection $\mc{Z}(Q_1,Q_2)$ of degree at most $D_1$ such that
\begin{equation}
\begin{split}
    &\Big\|\avprod_{i=1}^{2}|Ef_i|\Big\|_{L^{13/4}(B_R \cap N_{R^{1/2+\delta}}(\mc{Z}(P_1) ))}  
    \\&
    \lesssim 
    \Big\|\avprod_{i=1}^{2}|Ef_i|\Big\|_{L^{13/4}(B_R \cap N_{R^{1/2+\delta}}(\mc{Z}(Q_1,Q_2)))}. 
\end{split}
\end{equation}
If such $\mc{Z}(Q_1,Q_2)$ exists, then we apply the following lemma and prove \eqref{wallcase}.

\begin{lem}\label{onevariety}
For every pair of separated functions $g_1$ and $g_2$, one-dimensional transverse complete intersection $\mc{Z}(Q_1,Q_2)$ of degree at most $D_1$, it holds that
\begin{equation}
\begin{split}
    &\Big\|\avprod_{i=1}^{2}|Eg_i|\Big\|_{L^{13/4}(B_R \cap N_{R^{1/2+\delta}}(\mc{Z}(Q_1,Q_2)))} 
    \\&
\leq C_{\epsilon}R^{-c\delta \epsilon}R^{10\epsilon} \big(\avprod_{i=1}^{2}\|g_i\|_{L^2}\big)^{\frac{12}{13}+\epsilon}
\big(\avprod_{i=1}^{2}\max_{\theta \in \mathcal{P}(R^{-1/2}) }\|g_i\|_{L^{2}_{\mathrm{avg}}(\theta)}\big)^{\frac{1}{13}-\epsilon}
\end{split}
\end{equation}
under the induction  hypothesis \eqref{inductionhypothesis}.
\end{lem}

Let us postpone the proof of the lemma to the next section and consider the case that such $\mc{Z}(Q_1,Q_2)$ does not exist. The advantage of this case is that we can control tangent spaces of a variety. Let us explain more.

Let $\gamma$ be a fixed constant smaller than the implied constant in \eqref{separated}. This constant $\gamma$ is independent of all the parameters, for example, $\epsilon$, $\delta$, and $R$. Recall that since $P_1$ is a transverse complete intersection, the tangent planes are well-defined at every point of the variety.
We say that a ball $B(x_0,R^{1/2+\delta}) \subset N_{R^{1/2+\delta}}(\mc{Z}(P_1)) \cap B_R$ is \textit{regular} if on each connected component of $\mc{Z}(P_1) \cap B(x_0,10R^{1/2+\delta})$ the tangent space $T(\mc{Z}(P_1))$ is constant up to angle $\gamma$. For a regular ball $B$, we pick a point $z \in B \cap \mc{Z}(P_1)$ and define $V_B$ to be the two-dimensional tangent plane $T_z (\mc{Z}(P_1))$. It is proved on page 126 of \cite{guth2018} (see also page 16 of \cite{bejenaru2020optimal}) that, by the assumption that such $\mc{Z}(Q_1,Q_2)$ does not exist, there exists a 2-dimensional plane $V$ such that
\begin{equation}\label{regularball}
    \Big\|\avprod_{i=1}^{2}|Ef_i|\Big\|_{L^{13/4}(B_R \cap N_{R^{1/2+\delta}}(\mc{Z}(P_1) ))}
    \lesssim \Big\|\avprod_{i=1}^{2}|Ef_i|\Big\|_{L^{13/4}(\bigcup_{B \in \mathcal{B}_V}B)},
\end{equation}
where $\mathcal{B}_V$ is the set of regular balls such that the angle between $V_B$ and $V$ is smaller than $\gamma$. Since this statement is proved several times in the literature, we omit the details.

For simplicity, we introduce the notation 
\begin{equation}
N_1:=\bigcup_{B \in \mathcal{B}_V}B \subset B_R \cap N_{R^{1/2+\delta}}(\mc{Z}(P_1)).
\end{equation} 
Define
\begin{equation}\label{highlow}
\begin{split}
    &\T_{\geq 4\gamma}:=\{T \in \T: \mathrm{Angle}(v(T),V) \geq 4\gamma  \},
    \\&
    \T_{< 4\gamma}:=\{T \in \T: \mathrm{Angle}(v(T),V) < 4\gamma  \}.
\end{split}
\end{equation}
We split functions $Ef_i$ into three parts:
\begin{equation}
    Ef_i= Ef_{i,\geq 4\gamma}+Ef_{i,<4\gamma} +\mathrm{RapDec}(R)\|f_i\|_2,
\end{equation}
where
\begin{equation}
    f_{i,\geq 4\gamma}:=\sum_{T \in \T_{\geq 4\gamma}}f_{i,T}, \;\;\;
     f_{i,<4\gamma}:=\sum_{T \in \T_{< 4\gamma} }f_{i,T}.
\end{equation}
By the triangle inequality, the right hand side of \eqref{regularball} is bounded by 
\begin{equation}
    \begin{split}
      & \lessapprox
       \||Ef_{1, \geq 4 \gamma}Ef_{2, \geq 4\gamma}|^{1/2}\|_{L^{13/4}(N_1)}
      \\&+
       \||Ef_{1, < 4 \gamma}Ef_{2, \geq 4\gamma}|^{1/2}\|_{L^{13/4}(N_1)}
       \\&+
        \||Ef_{1, \geq 4 \gamma}Ef_{2, < 4\gamma}|^{1/2}\|_{L^{13/4}(N_1)}
        \\&+
         \||Ef_{1, < 4 \gamma}Ef_{2, < 4\gamma}|^{1/2}\|_{L^{13/4}(N_1)}.
    \end{split}
\end{equation}
We say that we are in \textit{a high angle dominant case} if the first three terms dominate the last term. Otherwise, we say that we are in \textit{a low angle dominant case}.

\subsection{The high angle dominant case}\label{highangle}
In this case, by the symmetric role of $f_1$ and $f_2$ and the $L^2$-orthogonality, the desired estimate \eqref{wallcase} follows from
\begin{equation}\label{highanglecase}
    \begin{split}
        &\||Ef_{1, \geq 4 \gamma}Ef_{2}|^{1/2}\|_{L^{13/4}(N_1)}
        \\& \leq
        CC_{\epsilon}R^{10\epsilon} \big(\avprod_{i=1}^{2}\|f_i\|_{L^2}\big)^{\frac{12}{13}+\epsilon}
\big(\avprod_{i=1}^{2}\max_{\theta \in \mathcal{P}(R^{-1/2}) }\|f_i\|_{L^{2}_{\mathrm{avg}}(\theta)}\big)^{\frac{1}{13}-\epsilon}
    \end{split}
\end{equation}
for some small constant $C$.
\medskip

Let us consider two subcases according to whether there exists a one-dimensional transverse complete intersection $\mc{Z}(Q_1,Q_2)$ of degree at most $D_1$ satisfying
\begin{equation}
    \||Ef_{1, \geq 4\gamma}Ef_2|^{1/2}\|_{L^{13/4}(N_1)}  \lesssim 
    \||Ef_{1, \geq 4\gamma}Ef_2|^{1/2}\|_{L^{13/4}(B_R \cap N_{R^{1/2+\delta}}(\mc{Z}(Q_1,Q_2)))}. 
\end{equation}
If such $\mc{Z}(Q_1,Q_2)$ exists, then we apply Lemma \ref{onevariety}, and by the $L^2$-orthogonality, we obtain \eqref{highanglecase}. Hence, we may assume that such $\mc{Z}(Q_1,Q_2)$ does not exist.
\medskip

Recall that $N_1 \subset B_R \cap N_{R^{1/2+\delta}}(\mc{Z}(P_1))$. 
We apply the polynomial partitioning lemma (Lemma \ref{polypartitioning}) to the function $\chi_{N_1}|Ef_{1, \geq 4\gamma}Ef_2|^{1/2}$ with the degree $D_1$. Then the second case of Lemma \ref{polypartitioning} cannot happen. Thus, there exist a polynomial $P_2:\R^3 \rightarrow \R$ of degree at most $D_1$ such that
\begin{equation}
\R^3 \setminus \mc{Z}(P_2) = \bigsqcup_{k=1}^{M} \widetilde{O}_k',
\end{equation}
where $M \simeq (D_1)^2$, and $\widetilde{O}_k'$ is a connected component of $\R^3 \setminus \mc{Z}(P_2)$, 
and if we define the cells $\widetilde{O}_k:=B_R \cap \big(O_k' \setminus N_{R^{1/2+\delta}}(\mc{Z}(P_2))\big)$, then
\begin{equation}\label{412}
\||Ef_{1, \geq 4\gamma}Ef_2|^{1/2}\|_{L^{13/4}( N_1)}^{13/4} \simeq (D_1)^2\||Ef_{1,\geq 4\gamma}Ef_2|^{1/2}\|_{L^{13/4}(N_1 \cap \widetilde{O}_{k} )}^{13/4}
\end{equation}
for $(D_1)^2$ many cells $\widetilde{O}_k$. By abusing the notation, we pretend that the cells $\widetilde{O}_k$ always satisfy the above inequality.

Define $\T_{\geq 4\gamma, k}$ by a sub-collection of the tubes in $\T_{\geq 4\gamma}$ that intersect $\widetilde{O}_k$
and $\T_k$ by a sub-collection of the tubes in $\T$ that intersect $\widetilde{O}_k$.
Since each tube $T \in \T$ can pass through at most $D_1+1$ many $\widetilde{O}_k$, we know that
\begin{equation}
    \sum_{k}\Big\|\sum_{T \in \T_{k} }f_{2,T}\Big\|_{L^2}^2 \lesssim D_1 \|f_2\|_2^2.
\end{equation}
As observed in \cite{bejenaru2020optimal}, each $T \in \T_{\geq 4\gamma}$ can intersect $\widetilde{O}_{k} \cap N_1$ at most $O(D^3)$ times. This is because $T \cap \mc{Z}(P_1)$ is contained in at most $O(D^3)$ balls of radius $R^{1/2+\delta}$ (see Lemma 5.7 of \cite{guth2018}). By this observation and the $L^2$-orthogonality, we obtain
\begin{equation}
    \sum_{k}\Big\|\sum_{T \in \T_{\geq 4\gamma, k} }f_{1,T}\Big\|_{L^2}^2 \lesssim D^3 \|f_1\|_2^2.
\end{equation}
By pigeonholing, we can choose $k_0$ such that
\begin{equation}
    \Big\|\sum_{T \in \T_{k_0} }f_{2,T}\Big\|_{L^2}^2 \lesssim D_1^{-1} \|f_2\|_2^2, \;\;\; 
    \Big\|\sum_{T \in \T_{\geq 4\gamma, k_0} }f_{1,T}\Big\|_{L^2}^2 \lesssim D^3D_1^{-2} \|f_1\|_2^2.
\end{equation}
Therefore, if we use \eqref{412} with $k=k_0$ and apply the induction hypothesis \eqref{inductionhypothesis}, by the above inequalities, we have
\begin{equation}
    \begin{split}
        &\||Ef_{1, \geq 4\gamma}Ef_2|^{1/2}\|_{L^{13/4}(N_1)} \\&\lesssim D^{\frac{9}{13}+\epsilon}D_1^{-\frac{1}{13}}
        C_{\epsilon}R^{10\epsilon} \big(\avprod_{i=1}^{2}\|f_i\|_{L^2}\big)^{\frac{12}{13}+\epsilon}
\big(\avprod_{i=1}^{2}\max_{\theta \in \mathcal{P}(R^{-1/2}) }\|f_i\|_{L^{2}_{\mathrm{avg}}(\theta)}\big)^{\frac{1}{13}-\epsilon}.
    \end{split}
\end{equation}
It suffices to take $D_1$ large enough compared to the constant $D$.

\subsection{The low angle dominant case}\label{lowanglecase1}
In this case, we prove the following.

\begin{equation}\label{lowanglecase}
    \begin{split}
        &\Big\|\avprod_{i=1}^{2}|Ef_{i, < 4 \gamma}|\Big\|_{L^{13/4}(N_1)}
        \\& \leq
        C_{\epsilon}R^{-c\epsilon\delta}R^{10\epsilon} \big(\avprod_{i=1}^{2}\|f_i\|_{L^2}\big)^{\frac{12}{13}+\epsilon}
\big(\avprod_{i=1}^{2}\max_{\theta \in \mathcal{P}(R^{-1/2}) }\|f_i\|_{L^{2}_{\mathrm{avg}}(\theta)}\big)^{\frac{1}{13}-\epsilon}.
    \end{split}
\end{equation}
One advantage of working with the wave packets with low angles is that it allows for an $L^4$-estimate as good as in \cite{MR3454378}. This is one observation already appeared in \cite{bejenaru2020optimal}. We will make use of it later.
\\

Recall that $N_1$ is a subset of $B_R \cap N_{R^{1/2+\delta} }(\mc{Z}(P_1))$. For simplicity, we define $W:=B_R \cap N_{R^{1/2+\delta} }(\mc{Z}(P_1))$.
We decompose the ball $B_R$ into smaller balls $B_j$ of radius $R^{1-\delta}$. For each ball $B_j$, we define transverse and tangential tubes as in \cite{MR3454378}.

\begin{defi}[Tangential tubes]
$\T_{j,-}$ is the set of all $T \in \T_{< 4\gamma}$ obeying the following  two conditions.
\begin{itemize}
    \item $T \cap W \cap B_j \neq \emptyset$
    \item If $z$ is any point of $\mc{Z}(P_1)$ lying in $2B_j \cap 10T$, then
    \begin{equation}
        \mathrm{Angle}(v(T), T_z (\mc{Z}(P_1))) \leq R^{-1/2+2\delta}.
    \end{equation}
\end{itemize}
Here $T_z (\mc{Z}(P_1))$ denotes the tangent space of $Z$ at the point $z$.
\end{defi}

\begin{defi}[Transverse tubes]
$\T_{j,+}$ is the set of all $T \in \T_{< 4\gamma}$ obeying the following two conditions.
\begin{itemize}
    \item $T \cap W \cap B_j \neq \emptyset$
    \item There exists a point of $\mc{Z}(P_1)$ lying in $2B_j \cap 10T$, so that
    \begin{equation}
        \mathrm{Angle}(v(T), T_z (\mc{Z}(P_1))) > R^{-1/2+2\delta}.
    \end{equation}
\end{itemize}
\end{defi}
 Notice that any tube $\T_{< 4\gamma}$ that intersects $W \cap B_j$ lies in exactly one of $\T_{j,+}$ and $\T_{j,-}$. Thus, on the set $W \cap B_j$,
\begin{equation}
    Ef_{i,<4\gamma}= \sum_{T \in \T_{j,+} } Ef_{1,T} +\sum_{T \in \T_{j,-} } Ef_{1,T} + \mathrm{RapDec}(R)\|f\|_2.
\end{equation}
For simplicity, we define $f_{i,j,+}:=\sum_{T \in \T_{j,+}}f_{i,T}$ and define $f_{i,j,-}$ similarly. 

We decompose $N_1$ into smaller parts $B_{j} \cap N_1$ and bound the left hand side of \eqref{lowanglecase} by
\begin{equation}\label{fourterms}
\begin{split}
    &\lessapprox
    \Big(\sum_{B_j}\||Ef_{1,j,+}Ef_{2,j,+}|^{1/2}\|_{L^{13/4}(N_1 \cap B_j)}^{13/4} \Big)^{\frac{4}{13}} \\&+\Big(\sum_{B_j}\||Ef_{1,j,-}Ef_{2,j,+}|^{1/2}\|_{L^{13/4}(N_1 \cap B_j)}^{13/4}\Big)^{\frac{4}{13}}
    \\&+\Big(\sum_{B_j}\||Ef_{1,j,+}Ef_{2,j,-}|^{1/2}\|_{L^{13/4}(N_1 \cap B_j)}^{13/4}\Big)^{\frac{4}{13}}
     \\&+\Big(\sum_{B_j}\||Ef_{1,j,-}Ef_{2,j,-}|^{1/2}\|_{L^{13/4}(N_1 \cap B_j)}^{13/4}\Big)^{\frac{4}{13}}.
\end{split}
\end{equation}
If the first term dominates the others, we say that we are in \textit{a transverse wall case}. Otherwise, we say that we are in \textit{a tangential wall case}.

\subsubsection{The transverse wall case}\label{transversewall}
In this case, we do not use much information on $N_1$. In this subsubsection, the constant $C$ may vary from line-to-line. This constant is independent of all the parameters, for example, $\epsilon$, $\delta$, and $R$. We start with the following bound.
\begin{equation}\label{tw:step1}
     \||Ef_{1, < 4 \gamma}Ef_{2,<4\gamma}|^{1/2}\|_{L^{13/4}(N_1)}^{13/4}
     \lesssim \sum_{B_j}\||Ef_{1,j,+}Ef_{2,j,+}|^{1/2}\|_{L^{13/4}(W \cap B_j)}^{13/4}.
\end{equation}
This case can be dealt with by following the argument in \cite{MR3454378}  line by line. Let us give the details. We apply the induction hypothesis \eqref{inductionhypothesis}  and the right hand side of \eqref{tw:step1} is bounded by
\begin{equation}\label{tw:step2}
CC_{\epsilon}^{\frac{13}{4}}
R^{\frac{13}{4} \cdot (1-\delta)10\epsilon}
\sum_{B_j}
\big(\avprod_{i=1}^{2}\|f_{i,j,+}\|_{L^2}\big)^{3+\frac{13\epsilon}{4}}
\big(\avprod_{i=1}^{2}\max_{\theta}\|f_{i,j,+}\|_{L^{2}_{\mathrm{avg}}(\theta)}\big)^{\frac{1}{4}-\frac{13\epsilon}{4}}.
\end{equation}
Next we apply the $L^2$-orthogonality, and replace $\|f_{i,j,+}\|_{L^2_{\mathrm{avg}}(\theta)}$ by $\|f_{i}\|_{L^2_{\mathrm{avg}}(\theta)}$. By the relation $\|\cdot\|_{l^{3+13\epsilon/4}} \leq \| \cdot \|_{l^2}$ and Cauchy-Schwarz inequality, \eqref{tw:step2} is further bounded by 
\begin{equation}\label{tw:step3}
    CC_{\epsilon}^{\frac{13}{4}}R^{-\frac{65}{2} \cdot \delta\epsilon}R^{\frac{13}{4} \cdot 10\epsilon }
\big(\avprod_{i=1}^{2}
\big(
\sum_{B_j}\|f_{i,j,+}\|_{L^2}^2 \big)^{\frac12(3+\frac{13\epsilon}{4})}\big)
\big(\avprod_{i=1}^{2}\max_{\theta}\|f_{i}\|_{L^{2}_{\mathrm{avg}}(\theta)}\big)^{\frac{1}{4}-\frac{13\epsilon}{4}}.
\end{equation}
By Lemma 5.7 of \cite{guth2018}, each tube $T \in \T$ belongs to $\T_{j,+}$ at most $O(D^3)$ many $j$ (see also Lemma 3.5 of \cite{MR3454378}). Hence, by the $L^2$-orthogonality, as in \eqref{38}, we obtain
\begin{equation}
    \sum_{B_j}\|f_{i,j,+}\|_{L^2}^2 \lesssim D^3 \|f_{i}\|_2^2.
\end{equation}
Therefore, \eqref{tw:step3} is bounded by
\begin{equation}
    CC_{\epsilon}^{\frac{13}{4}}D^{10}
    R^{-\frac{65}{2} \cdot \delta\epsilon}R^{\frac{13}{4} \cdot 10\epsilon}
\Big(\avprod_{i=1}^{2}
\|f_{i}\|_{L^2}\Big) ^{3+\frac{13\epsilon}{4}}
\big(\avprod_{i=1}^{2}\max_{\theta}\|f_{i}\|_{L^{2}_{\mathrm{avg}}(\theta)}\big)^{\frac{1}{4}-\frac{13\epsilon}{4}}.
\end{equation}
It suffices to note that $D$ is a fixed number independent of $R$ and we were able to assume that $R$ is large enough compared to $D$ by the base of the induction.

\subsubsection{The tangential wall case} 
In this case, the first term in \eqref{fourterms} is bounded by the other terms. Therefore, it suffices to prove the following proposition.
\begin{prop}\label{mainestimate}
For every pair of separated functions $g_1$ and $g_2$, and ball $B_j$, it holds that
\begin{equation} 
\begin{split}
&\||Eg_{1,j,-}Eg_{2, <4 \gamma}|^{\frac12}\|_{L^{13/4}(B_j \cap N_1)} 
\\&
\leq C_{\epsilon}R^{O(\delta)} \big(\avprod_{i=1}^{2}\|g_i\|_{L^2}\big)^{\frac{12}{13}}
\big(\avprod_{i=1}^{2}\max_{\theta}\|g_i\|_{L^{2}_{\mathrm{avg}}(\theta)}\big)^{\frac{1}{13}}.
\end{split}
\end{equation}
\end{prop}

Recall that $N_1$ is a subset of $W=B_R \cap N_{R^{1/2+\delta}}(\mc{Z}(P_1))$ and $\mc{Z}(P_1)$ is a two-dimensional complete intersection of degree at most $D$.

The proof of the above proposition is the main part of this paper.
We start with the observation on page 27 of \cite{MR4205111}: For every tube $T \in \T$, there exist some tubes $\widetilde{T}_{T,m}$ of dimension $5R^{1/2+\delta} \times 5R^{1/2+\delta} \times l_m$ for some $l_m \geq R^{1/2+\delta}$ such that
\begin{equation}\label{zahlobservation}
    T \cap N_{R^{1/2+\delta}}(\mc{Z}(P_1)) \subset \bigsqcup_{m=1}^{\lesssim C_{D}} \widetilde{T}_{T,m} \subset N_{20R^{1/2+\delta}}(\mc{Z}(P_1)),
\end{equation}
and
\begin{equation}\label{distance}
    \mathrm{dist}(\widetilde{T}_{T,m},\widetilde{T}_{T,m'}) \geq 2R^{1/2+\delta}
\end{equation}
for any $m,m'$. The property \eqref{distance} is not crucial, but it helps to avoid some technical issue.
Let us prove the observation. By Theorem \ref{bpr96}, there are at most $C_{D}$ many connected components of $N_{R^{1/2+\delta}}(T) \cap \mc{Z}(P_1)$. We take the smallest union of subtubes of $N_{R^{1/2+\delta}}(T)$ such that the union covers 
$N_{R^{1/2+\delta}}(T) \cap \mc{Z}(P_1)$. We slightly enlarge each subtube so that their union covers $T \cap N_{R^{1/2+\delta}}(\mc{Z}(P_1))$. Note that each subtube is contained in $N_{20R^{1/2+\delta}}(\mc{Z}(P_1))$ by the construction of the subtubes. The distance condition \eqref{distance} can be easily attained by modifying the subtubes. This completes the proof of the observation.
\medskip

We take the characteristic function $\chi_{\widetilde{T}_{T,m}}$ of a tube $\widetilde{T}_{T,m}$. By the observation \eqref{zahlobservation}, it holds that 
\begin{equation}\label{identity}
    Eg_{2, <4\gamma}(x) = \sum_{T \in \T_{<4\gamma} }\sum_{m=1}^{\lesssim C_{D}} \chi_{\widetilde{T}_{T,m}}(x) Eg_{2,T}(x) +\mathrm{RapDec}(R)\|g_2\|_2
\end{equation}
for every $ x \in B_j \cap W$.
Note that the above identity does not need to be true outside of $B_j \cap W$. Using \eqref{identity}, we obtain
\begin{equation}
\begin{split}
    &\||Eg_{1,j,-}Eg_{2,<4\gamma}|^{1/2}\|_{L^{13/4}(B_j \cap N_1)} \\&
    \lessapprox
    \Big\||Eg_{1,j,-}|^{1/2}
    \big|
    \sum_{T \in \T_{<4\gamma} }\sum_{m=1}^{\lesssim C_{D}} \chi_{\widetilde{T}_{T,m}}Eg_{2,T}
    \big|^{1/2}\Big\|_{L^{13/4}(B_j \cap N_1)}.
\end{split}
\end{equation}
By the triangle inequality and taking the maximum, the right-hand side above is bounded by 
\begin{equation}
    C(C_D)^{1/2}
    \Big\||Eg_{1,j,-}|^{1/2}
    \big|
    \sum_{T \in \T_{<4\gamma} } \chi_{\widetilde{T}_{T,m}}Eg_{2,T}
    \big|^{1/2}\Big\|_{L^{13/4}(B_j \cap N_1)}
\end{equation}
for some $m$. For simplicity, let us use the notation $\chi_{\widetilde{T}_{T}}$ for $\chi_{\widetilde{T}_{T,m}}$.
Recall that the length of the longest direction of the tube $\widetilde{T}_{T}$ is greater than $R^{1/2+\delta}$ and smaller than $R$. Thus, by a dyadic pigeonhling and taking a sub-collection, we may assume that the longest directions of all the nonempty tubes $\widetilde{T}_{T}$ are comparable and we denote the length by $l$. Recall that the constant $D$ is independent of the parameter $R$.
Hence, what we need to prove becomes
\begin{equation}\label{330}
\begin{split}
    &\Big\||Eg_{1,j,-}|^{1/2}
    \big|
    \sum_{T \in \T_{<4\gamma} } \chi_{\widetilde{T}_{T}}Eg_{2,T}
    \big|^{1/2}\Big\|_{L^{13/4}(B_j \cap N_1)}
    \\&\lesssim
    R^{O(\delta)} \big(\avprod_{i=1}^{2}\|g_i\|_{L^2}\big)^{\frac{12}{13}}
\big(\avprod_{i=1}^{2}\max_{\theta}\|g_i\|_{L^{2}_{\mathrm{avg}}(\theta)}\big)^{\frac{1}{13}},
\end{split}
\end{equation}
where $\widetilde{T}_T$ has a longest direction with length 0 or $l$.
\\

We will interpolate the $L^2$ estimate and the $L^4$ estimate by H\"{o}lder's inequality:

\begin{equation}\label{holder}
\begin{split}
    &\Big\||Eg_{1,j,-}|^{\frac12}
    \big|
    \sum_{T \in \T_{<4\gamma} } \chi_{\widetilde{T}_{T}}Eg_{2,T}
    \big|^{\frac12}\Big\|_{L^{13/4}(B_j \cap N_1)}
    \\&
    \lesssim
    \Big\||Eg_{1,j,-}|^{\frac12}
    \big|
    \sum_{T \in \T_{<4\gamma} } \chi_{\widetilde{T}_{T}}Eg_{2,T}
    \big|^{\frac12}
    \Big\|_{L^{2}(B_j \cap N_1)}^{\frac{3}{13}}
    \\& \times
    \Big\||Eg_{1,j,-}|^{\frac12}
    \Big|
    \sum_{T \in \T_{<4\gamma} } \chi_{\widetilde{T}_{T}}Eg_{2,T}
    \big|^{\frac12}\Big\|_{L^{4}(B_j \cap N_1)}^{\frac{10}{13}}.
    \end{split}
\end{equation}

Let us first estimate the $L^4$-norm. We define $\Theta_{\mathrm{Leng}(l)}$ by the collection of directions of the tubes $T \in \T$ for which the intersection of the tube and $W \cap B_j$ contains a tube of  dimensions $R^{1/2+\delta} \times R^{1/2+\delta} \times l$. Define
\begin{equation}
    \T_{\mathrm{Leng}(l)}:= \{T \in \T_{<4\gamma}: v(T) \in \Theta_{\mathrm{Leng}(l)}\}, \;\;\;\; g_{2,\mathrm{Leng}(l)}:= \sum_{T \in \T_{\mathrm{Leng}(l)} }g_{2,T}.
\end{equation} 

We claim that
\begin{equation}\label{l4claim}
\begin{split}
    &\Big\||Eg_{1,j,-}|^{\frac12}
    \Big|
    \sum_{T \in \T_{<4\gamma} } \chi_{\widetilde{T}_{T}}Eg_{2,T}
    \big|^{\frac12}\Big\|_{L^{4}(B_j \cap N_1)}
    \\&
    \lessapprox R^{-\frac18+O(\delta)}\Big(\big(\sum_{T \in \T_{j,-} } \|g_{1,T}\|_2^2 \big)^{\frac12} \big( \sum_{T \in \T_{\mathrm{Leng}(l)}} \| g_{2,T}\|_2^2\big)^{\frac12}\Big)^{\frac12}.
\end{split}
\end{equation}
Recall that $B_j \cap W$ is a union of regular balls $Q$ of radius $R^{1/2}$. Define $\T_{j,-,Q}$ by the set of tubes in $\T_{j,-}$ intersecting $Q$ and define $\T_{\mathrm{Leng}(l),Q}$ similarly.  Note that on each set $Q \cap W \cap B_j$
\begin{equation}\label{pointwiseonQ}
    Eg_{1,j,-}=\sum_{T \in \T_{j,-,Q}}Eg_{1,T}+\mathrm{RapDec}(R)\|g_1\|_2.
\end{equation}
Similarly, on each set $Q \cap W \cap B_j$
\begin{equation}
    \sum_{T \in \T_{<4\gamma} } \chi_{\widetilde{T}_{T}}Eg_{2,T}=\sum_{T \in \T_{\mathrm{Leng}(l),Q,\sim}}Eg_{2,T}+\mathrm{RapDec}(R)\|g_2\|_2,
\end{equation}
 where 
 \begin{equation}
     \T_{\mathrm{Leng}(l),Q,\sim}:= \{ T \in \T_{\mathrm{Leng}(l),Q}: \widetilde{T}_T \cap Q \cap W \cap B_j \neq \emptyset   \}.
 \end{equation}
Here, the set $\T_{<4\gamma}$ was defined in \eqref{highlow}, $B_j$ is a ball of radius $R^{1-\delta}$, and $\widetilde{T}_T$ is a sub-tube of $T \in \T_{<4\gamma}$ the longest length of which is $0$ or $l$. By the above identities, we know that 
\begin{equation}\label{claimstep2}
\begin{split}
   &\Big\||Eg_{1,j,-}|^{\frac12}
    \Big|
    \sum_{T \in \T_{<4\gamma} } \chi_{\widetilde{T}_{T}}Eg_{2,T}
    \big|^{\frac12}\Big\|_{L^{4}(B_j \cap N_1)}
    \\& \lessapprox
    \Big(\sum_{Q: Q \cap W \cap B_j \neq \emptyset} \Big\|\big|\sum_{T \in \T_{j,-,Q}} Eg_{1,T} \sum_{T \in \T_{\mathrm{Leng}(l),Q,\sim}}
    Eg_{2,T}
    \big|^{\frac12}\Big\|_{L^{4}(Q)}^4\Big)^{\frac14}.
\end{split}
\end{equation}

We fix $Q$.
Since $Q \cap W \neq \emptyset$, we can choose a point $z \in \mc{Z}(P_1) \cap N_{R^{1/2+\delta}}(Q)$. Notice that by the definition of $\T_{j,-}$, every $T \in \T_{j,-,Q}$ satisfies that
$
    \mathrm{Angle}(v(T),T_z\mc{Z}(P_1)) \leq R^{-1/2+2\delta}$. Thus,  the function $\sum_{T \in \T_{j,-,Q}} g_{1,T}$ is supported on some strip of width $R^{-1/2+O(\delta)}$. 
    By a simple change of variables, we may assume that the strip is $[0,1] \times [0,R^{-1/2+O(\delta)}]$. On the other hand, by the definition, we know that $\T_{\mathrm{Leng}(l),Q,\sim} \subset \T_{<4\gamma}$, and thus, the support of $\sum_{T \in \T_{\mathrm{Leng}(l),Q,\sim}}g_{2,T}$ is contained in $[0,1] \times [0,100\gamma]$. Since $g_1$ and $g_2$ are separated and the constant $\gamma$ is much smaller than the implied constant in the definition of the separation \eqref{separated}, we can conclude that 
    \begin{equation}
        \mathrm{dist}\bigg(\pi \Big( \mathrm{supp}\big(\sum_{T \in \T_{j,-,Q}} g_{1,T} \big) \Big), 
        \pi \Big( \mathrm{supp}\big(\sum_{T \in \T_{\mathrm{Leng}(l),Q,\sim}}g_{2,T} \big) \Big)
        \bigg) \gtrsim 1,
    \end{equation}
    where $\pi:\R^2 \rightarrow \R$ is a projection map defined as $\pi(\xi_1,\xi_2):=\xi_1$.
    Therefore, we can perform the the standard $L^4$-argument (see Lemma 3.10 of \cite{MR3454378}), or simply apply Theorem 1.3 of \cite{bejenaru2019multilinear}, and obtain
    \begin{equation}\label{claimstep2.5}
    \begin{split}
        &\Big\|\big|\sum_{T \in \T_{j,-,Q}} Eg_{1,T} \sum_{T \in \T_{\mathrm{Leng}(l),Q,\sim}}
    Eg_{2,T}
    \big|^{\frac12}\Big\|_{L^{4}(Q)}
    \\& \lesssim
    R^{-\frac18+O(\delta)}
    \Big(
    \big(\sum_{T \in \T_{j,-,Q} } \|g_{1,T}\|_2^2 \big)^{\frac12} \big( \sum_{T \in \T_{\mathrm{Leng}(l),Q,\sim}} \| g_{2,T}\|_2^2\big)^{\frac12}\Big)^{\frac12}.
    \end{split}
    \end{equation}
We bound the sum over $\T_{\mathrm{Leng}(l),Q,\sim}$ by that over $\T_{\mathrm{Leng}(l),Q}$.
Define a function $\chi(T,Q)$ whose value is 1 if $T$ and $Q$ intersect, and 0 otherwise. 
Notice that
\begin{equation}
    \sum_{T \in \T_{j,-,Q} }\|g_{1,T}\|_2^2
    = \sum_{T \in \T_{j,-} }
    \chi(T,Q)
    \|g_{1,T}\|_2^2.
\end{equation}
We have a similar property for $\T_{\mathrm{Leng}(l),Q}$.
Notice that for every $T_1$ and $T_2$ whose direction is separated by $\simeq 1$ it holds that 
\begin{equation}
    \sum_Q \chi(T_1,Q)\chi(T_2,Q) \lesssim R^{O(\delta)}.
\end{equation}
By this inequality,
we obtain
\begin{equation}
\begin{split}
    &\sum_Q \sum_{T \in \T_{j,-,Q} }\|g_{1,T}\|_2^2 \sum_{T \in \T_{\mathrm{Leng}(l),Q} }\|g_{2,T}\|_2^2
    \\&
    \lesssim 
    R^{O(\delta)}
     \sum_{T_1 \in \T_{j,-} }
     \sum_{T_2 \in \T_{\mathrm{Leng}(l)} }
    \|g_{1,T_1}\|_2^2 
    \|g_{2,T_2}\|_2^2.
\end{split}
\end{equation}
Therefore, by  \eqref{claimstep2.5} and the above inequality, we know that
\begin{equation}\label{claimstep3}
\begin{split}
    &\sum_Q \Big\|\big|\sum_{T \in \T_{j,-,Q}} Eg_{1,T} \sum_{T \in \T_{\mathrm{Leng}(l),Q,\sim}}
    Eg_{2,T}
    \big|^{\frac12}\Big\|_{L^{4}(Q)}^4
    \\& \lesssim
    R^{-\frac12+O(\delta)}
    \sum_{T \in \T_{j,-} } \|g_{1,T}\|_2^2
    \sum_{T \in \T_{\mathrm{Leng}(l)}} \| g_{2,T}\|_2^2.
    \end{split}
\end{equation}
The claim \eqref{l4claim} follows by combining  \eqref{claimstep2} and \eqref{claimstep3}.
\medskip

Let us move on to the $L^2$-estimate. A main estimate is the following. 

\begin{lem}\label{L2estimate}
\begin{equation}
\|\sum_{T \in \T_{< 4\gamma} } \chi_{\widetilde{T}_{T}}Eg_{2,T}\|_{L^2(B_j \cap W)}^2 \lesssim R^{O(\delta)} \big(\frac{l}{R} \big)
\sum_{T \in \T_{\mathrm{Leng}(l)}}
\|
 Eg_{2,T}\|_{L^2(w_{B_j})}^2.
\end{equation}
\end{lem}

One may think of this estimate as a counterpart of \cite[eq. (4.16)]{MR4205111} in the restriction problem setting.
\medskip

We assume this lemma for a moment and finish the proof of Proposition \ref{mainestimate}. By the Cauchy-Schwarz inequality, 
\begin{equation}
\begin{split}
    &\Big\||Eg_{1,j,-}|^{\frac12}
    \big|
    \sum_{T \in \T_{< 4\gamma} } \chi_{\widetilde{T}_{T}}Eg_{2,T}
    \big|^{\frac12}
    \Big\|_{L^{2}(B_j \cap N_1)} 
    \\&
    \lesssim \Big\|Eg_{1,j,-}
    \Big\|_{L^{2}(B_j)}^{\frac12}
    \Big\|
    \sum_{T \in \T_{< 4\gamma} } \chi_{\widetilde{T}_{T}}Eg_{2,T}
    \Big\|_{L^{2}(B_j \cap W)}^{\frac12}.
    \end{split}
\end{equation}
By Lemma \ref{L2estimate}, the above term is bounded by
\begin{equation}\label{451}
R^{O(\delta)}
    \big(\frac{l}{R}\big)^{\frac14}
    \|Eg_{1,j,-}
    \|_{L^{2}(B_j)}^{\frac12}
    \big(
    \sum_{T \in \T_{\mathrm{Leng}(l)}} 
    \|
     Eg_{2,T}
    \|_{L^{2}(w_{B_j} )}^2\big)^{\frac14}.
\end{equation}
By the standard $L^2$-estimate 
\begin{equation}\label{standardl2}
    \|Eg\|_{L^2(B_j)}^2 \lesssim R^{1-\delta} \|g\|_2^2,
\end{equation}
the term \eqref{451} is further bounded by
\begin{equation}
R^{O(\delta)}
    \big(\frac{l}{R}\big)^{\frac14}(R^{1-\delta})^{\frac12}
    \big( \sum_{T \in \T_{j,-}} \|g_{1,T}\|_2^2\big)^{\frac14}
    \big( \sum_{T \in \T_{\mathrm{Leng}(l)} }
    \| g_{2,T}\|_2^2\big)^{\frac14}.
\end{equation}

By \eqref{holder}, the $L^4$-estimate \eqref{l4claim}, and the $L^2$-estimate, we obtain
\begin{equation}
    \begin{split}
       & \Big\||Eg_{1,j,-}|^{\frac12}
    \big|
    \sum_{T \in \T_{<4\gamma} } \chi_{\widetilde{T}_{T}}Eg_{2,T}
    \big|^{\frac12}\Big\|_{L^{13/4}(B_j \cap N_1)}
    \\&\lessapprox R^{-\frac{1}{26}+O(\delta)}l^{\frac{3}{52}} \big( \sum_{T \in \T_{j,-}} \|g_{1,T}\|_2^2\big)^{\frac14}
    \big(\sum_{T \in \T_{\mathrm{Leng}(l) }}\| g_{2,T}\|_2^2\big)^{\frac14}.
    \end{split}
\end{equation}
We now apply Lemma \ref{polylemma}\footnote{By Wongkew's theorem \cite{MR1211391}, one can see that the assumption \eqref{wongkew} is satisfied and we can apply the lemma.} to the function $g_{1,j,-}$ with $n=3$, $r=R$, and $\rho=R^{1/2}$, by the $L^2$-orthogonality, we obtain
\begin{equation}
\begin{split}
    \sum_{T \in \T_{j,-}} \|g_{1,T}\|_2^2 &\lesssim \sum_{\theta}\sum_{T \in \T_{j,-}} \big\|g_{1,T}\big\|_{L^2(\theta)}^2 
    \\&
    \lesssim R^{1/2+O(\delta)}\max_{\theta}
    \Big(
    \sum_{T \in \T(\theta)}
    \big\|g_{1,T}\big\|_{L^2(\theta)}^2\Big) 
    \\&
    \lesssim R^{-1/2+O(\delta)}
    \max_{\theta}
    \big\|g_{1}\big\|_{L^2_{\mathrm{avg}}(\theta)}^2,
\end{split}
\end{equation}
and similarly, 
\begin{equation}
    \sum_{T \in \T_{\mathrm{Leng}(l)}} \| g_{2,T}\|_2^2
    \lesssim
    l^{-1}R^{1/2+O(\delta)}\max_{\theta} \|g_2\|_{L^2_{\mathrm{avg} }(\theta)}^2.
\end{equation}
Therefore, by combining these estimates with the $L^2$-orthogonality, we conclude that
\begin{equation}
    \begin{split}
       & \Big\||Eg_{1,j,-}|^{\frac12}
    \big|
    \sum_{T \in \T } \chi_{\widetilde{T}_{T}}Eg_{2,T}
    \big|^{\frac12}\Big\|_{L^{13/4}(B_j \cap N_1)}
    \\&\lessapprox R^{-\frac{1}{26}+O(\delta)}l^{\frac{1}{26}} \big(\avprod_{i=1}^{2}\|g_{i}\|_2\big)^{\frac12 \cdot \frac{12}{13} }
    (\avprod_{i=1}^{2}\max_{\theta}
    \| g_{i}\|_{L^2_{\mathrm{avg}}(\theta)}\big)^{\frac12 \cdot \frac{1}{13} }.
    \end{split}
\end{equation}
 Proposition \ref{mainestimate} follows by the upper bound $l \leq R$. \\

It remains to prove Lemma \ref{L2estimate}

\begin{proof}[Proof of Lemma \ref{L2estimate}] We cover $B_j \cap W$ by smaller balls $Q$ of radius $R^{1/2}$. By the $L^2$-orthogonality, we see that
\begin{equation}
    \Big\|\sum_{T \in \T_{<4\gamma} } \chi_{\widetilde{T}_{T}}Eg_{2,T}\Big\|_{L^2(Q)}^2 \lesssim R^{O(\delta)}
     \sum_{T \in \T_{\mathrm{Leng}(l)}: \widetilde{T}_T \cap Q \neq \emptyset }  \|  Eg_{2,T }\|_{L^2(w_Q)}^2.
\end{equation}
We sum over all the balls $Q$ intersecting  $B_j \cap W$ and obtain
\begin{equation}\label{348}
    \Big\|\sum_{T \in \T_{<4\gamma} } \chi_{\widetilde{T}_{T}}Eg_{2,T}\Big\|_{L^2(B_j \cap W)}^2 \lesssim R^{O(\delta)}
     \sum_{T \in \T_{\mathrm{Leng}(l)} }   \| Eg_{2,T}\|_{L^2( \sum_{Q: Q \cap \widetilde{T}_T \neq \emptyset } w_{Q})}^2.
\end{equation}
By an standard application of the essentially constant property (Lemma 6.4 of \cite{guth2018}), the right hand side of \eqref{348} is bounded by
\begin{equation}
R^{O(\delta)}
    \frac{|\widetilde{T}_T|}{|T \cap B_j|}
    \sum_{T \in \T_{\mathrm{Leng}(l)} }\|Eg_{2,T }\|_{L^2(w_{B_{j}})}^2.
\end{equation}
It suffices to recall that $\widetilde{T}_T$ has the dimension $5R^{1/2+\delta} \times 5R^{1/2+\delta} \times l$.
\end{proof}

\section{A proof of  theorem \ref{bilinearrestriction}: The remaining case}

In the previous two sections, we proved Proposition \ref{reducedmainestimate} by assuming Lemma \ref{onevariety}. 
In this section, we prove the lemma. Let us recall the lemma.
\begin{lem}\label{onevariety'}
For every pair of separated functions $g_1$ and $g_2$, one-dimensional transverse complete intersection $\mc{Z}(Q_1,Q_2)$ of degree at most $D_1$, it holds that
\begin{equation}\label{51}
\begin{split}
    &\Big\|\avprod_{i=1}^{2}|Eg_i|\Big\|_{L^{13/4}(B_R \cap N_{R^{1/2+\delta}}(\mc{Z}(Q_1,Q_2)))} 
    \\&
\leq C_{\epsilon}R^{-c\epsilon \delta}R^{10\epsilon} \big(\avprod_{i=1}^{2}\|g_i\|_{L^2}\big)^{\frac{12}{13}+\epsilon}
\big(\avprod_{i=1}^{2}\max_{\theta \in \mathcal{P}(R^{-1/2}) }\|g_i\|_{L^{2}_{\mathrm{avg}}(\theta)}\big)^{\frac{1}{13}-\epsilon}
\end{split}
\end{equation}
under the induction  hypothesis \eqref{inductionhypothesis}.
\end{lem}

The proof of this lemma shares some similarity to that for the low angle dominant case (Subsection \ref{lowanglecase1}). We only sketch the proof here. 
\\

Let us use the notation $Y:=B_R \cap N_{R^{1/2+\delta}}(\mc{Z}(Q_1,Q_2))$ for simplicity. We cover $Y$ by smaller balls ${B}_{j}$ of radius $R^{1-\delta}$.  Define transverse and tangential tubes with respect to the transverse complete intersection $\mc{Z}(Q_1,Q_2)$ as follows.

\begin{defi}[Tangential tube with respect to $\mc{Z}(Q_1,Q_2)$]
$\T_{j,\mathrm{tang} }$ is the set of all $T \in \T$ obeying the following  two conditions.
\begin{itemize}
    \item $T \cap {B}_j \cap Y \neq \emptyset$
    \item If $z$ is any  point of $\mc{Z}(Q_1,Q_2)$ lying in $2{B}_j \cap 10T$, then
    \begin{equation}
        \mathrm{Angle}(v(T), T_z (\mc{Z}(Q_1,Q_2))) \leq R^{-1/2+2\delta}.
    \end{equation}
\end{itemize}

\end{defi}

\begin{defi}[Transverse tube with respect to $\mc{Z}(Q_1,Q_2)$]
$\T_{j,\mathrm{trans} }$ is the set of all $T \in \T$ obeying the following two conditions.
\begin{itemize}
    \item $T \cap {B}_j \cap Y \neq \emptyset$
    \item There exists a  point of $\mc{Z}(Q_1,Q_2)$ lying in $2{B}_j \cap 10T$, so that
    \begin{equation}
        \mathrm{Angle}(v(T), T_z (\mc{Z}(Q_1,Q_2))) > R^{-1/2+2\delta}.
    \end{equation}
\end{itemize}
\end{defi}
Define $g_{i,j,\mathrm{trans} }:=\sum_{T \in \T_{j,\mathrm{trans} } }g_{i,T}$ and define $g_{i,j,\mathrm{tang} }$ similarly. Note that 
\begin{equation}
    Eg_i=Eg_{i,j,\mathrm{tang} }+Eg_{i,j,\mathrm{trans} }+\mathrm{RapDec}(R)\|g_i\|_2.
\end{equation}
By the triangle inequality, the left hand side of \eqref{51} is bounded by
\begin{equation}\label{fourterms'}
\begin{split}
    &\lessapprox \Big(\sum_{B_j}\||Eg_{1,j,\mathrm{trans} }Eg_{2,j,\mathrm{trans} }|^{1/2}\|_{L^{13/4}(Y \cap B_j)}^{13/4}\Big)^{\frac{4}{13}} \\&+\Big(\sum_{B_j}\||Eg_{1,j,\mathrm{tang} }Eg_{2,j,\mathrm{trans} }|^{1/2}\|_{L^{13/4}(Y \cap B_j)}^{13/4}\Big)^{\frac{4}{13}}
    \\&+\Big(\sum_{B_j}\||Eg_{1,j,\mathrm{trans} }Eg_{2,j,\mathrm{tang} }|^{1/2}\|_{L^{13/4}(Y \cap B_j)}^{13/4}\Big)^{\frac{4}{13}}
     \\&+\Big(\sum_{B_j}\||Eg_{1,j,\mathrm{tang} }Eg_{2,j,\mathrm{tang} }|^{1/2}\|_{L^{13/4}(Y \cap B_j)}^{13/4}\Big)^{\frac{4}{13}}.
\end{split}
\end{equation}

Let us consider the case that the first term dominates the others. By Lemma 5.7 of \cite{guth2018}, each tube $T \in \T$ belongs to $\T_{j,\mathrm{trans} }$ at most $O((D_1)^3)$ many $j$. By the $L^2$-orthogonality, this implies the following inequality:
\begin{equation}
    \sum_{B_j} \|g_{i,j,\mathrm{trans}}\|_2^2 \lesssim (D_1)^3 \|g_i\|_2^2.
\end{equation}
Hence, by following the arguments in the transverse wall case (Subsection \ref{transversewall}) line by line with $g_{i,j,\mathrm{trans}}$ replacing $f_{i,j,+}$, one can get the desired bound. Since the arguments are identical, we leave out the details.

Let us consider the case that the first term is dominated by the others. In this case, by replacing the summation by the maximum, it suffices to prove
\begin{equation}
\begin{split}
    &\||Eg_{1,j,\mathrm{tang} }Eg_{2}|^{1/2}\|_{L^{13/4}(Y \cap B_j)}
    \\&
    \lesssim R^{O(\delta)}
    \big(\avprod_{i=1}^{2}\|g_i\|_{L^2}\big)^{\frac{12}{13}}
\big(\avprod_{i=1}^{2}\max_{\theta \in \mathcal{P}(R^{-1/2}) }\|g_i\|_{L^{2}_{\mathrm{avg}}(\theta)}\big)^{\frac{1}{13}}
\end{split}
\end{equation}
for every function $g_1$ and $g_2$. By H\"{o}lder's inequality, the left hand side is bounded by a constant multiple of
\begin{equation}\label{57}
    \||Eg_{1,j,\mathrm{tang}}Eg_{2}|^{1/2}\|_{L^{2}(Y \cap B_j)}^{1/26}
    \||Eg_{1,j,\mathrm{tang}}Eg_{2}|^{1/2}\|_{L^{10/3}(Y \cap B_j)}^{25/26}.
\end{equation}
To treat the $L^2$-norm, we simply apply the Cauchy-Schwarz inequality and the standard $L^2$-estimate \eqref{standardl2}.
For the $L^{10/3}$-norm, we apply Tao's bilinear restriction estimate \cite{MR2033842}. Then the above term is bounded by
\begin{equation}\label{58}
    R^{1/52+O(\delta)}\|g_{1,j,\mathrm{tang}}\|_2^{1/2} \|g_2\|_2^{1/2}.
\end{equation}
By the polynomial Wolff axioms (Lemma \ref{polylemma}) with $n=3$, $r=R$, and $\rho=R^{1/2}$, we know that
\begin{equation}
    \|g_{1,j,\mathrm{tang}}\|_2^2 \lesssim R^{-1}R^{O(\delta)}\max_{\theta}\|g_1\|_{L^2_{\mathrm{avg}}(\theta)}^2.
\end{equation}
Therefore, \eqref{58} is bounded by
\begin{equation}
    R^{O(\delta)}
    \big(\avprod_{i=1}^{2}\|g_i\|_{L^2}\big)^{\frac{12}{13}}
\big(\avprod_{i=1}^{2}\max_{\theta \in \mathcal{P}(R^{-1/2}) }\|g_i\|_{L^{2}_{\mathrm{avg}}(\theta)}\big)^{\frac{1}{13}}
\end{equation}
and this completes the proof of the lemma.

	\bibliography{reference}{}
	\bibliographystyle{alpha}

\medskip

\medskip

\noindent Department of Mathematics, University of Wisconsin-Madison\\
\noindent Department of Mathematics, Massachusetts Institute of Technology\\
\emph{Email address}: changkeun.math@gmail.com

\end{document}